\documentclass{mystyle}

\title[Sphere Packing Proper Colorings of an Expander Graph]{Sphere Packing Proper Colorings of an Expander Graph}

\author[Honglin Zhu]{Honglin Zhu}
\begin{document}

\begin{abstract}
    We introduce graphical error-correcting codes, a new notion of error-correcting codes on $[q]^n$, where a code is a set of proper $q$-colorings of some fixed $n$-vertex graph $G$. We then say that a set of $M$ proper $q$-colorings of $G$ form a $(G, M, d)$ code if any pair of colorings in the set have Hamming distance at least $d$. This directly generalizes typical $(n, M, d)$ codes of $q$-ary strings of length $n$ since we can take $G$ as the empty graph on $n$ vertices. 

    We investigate how one-sided spectral expansion relates to the largest possible set of error-correcting colorings on a graph. For fixed $(\delta, \lambda) \in [0, 1] \times [-1, 1]$ and positive integer $d$, let $f_{\delta, \lambda, d}(n)$ denote the maximum $M$ such that there exists some $d$-regular graph $G$ on at most $n$ vertices with normalized second eigenvalue at most $\lambda$ that has a $(G, M, d)$ code. We study the growth of $f$ as $n$ goes to infinity. We partially characterize the regimes of $(\delta, \lambda)$ where $f$ grows exponentially or is bounded by a constant, respectively. We also prove several sharp phase transitions between these regimes.
\end{abstract}

\maketitle

\section{Introduction}
The study of error-correcting codes of a fixed length over some alphabet is a central topic in coding theory.  Given the word length $n$, the alphabet $[q]$, and the minimum distance $d$, the goal is to find a large subset of $[q]^n$ where the Hamming distance between any two codewords is at least $d$. Such a code of size $M$ is called an \emph{$(n, M, d)$ code}. This problem can be viewed as a sphere packing problem on $[q]^n$ equipped with the Hamming distance.

We extend this notion to introduce the sphere packing problem of proper $q$-colorings of a graph, which we call \emph{graphical error-correcting codes}. In particular, given an $n$ vertex graph $G$, we identify the set of vertex colorings of $G$ using $q$ colors with $[q]^n$. Now, given a minimum distance $d$, we would like to find a large set of vertex colorings where the Hamming distance between any two colorings is at least $d$, with the extra requirement that each coloring is a proper coloring of the graph $G$. Such a set of proper $q$-colorings of size $M$ is called a \emph{$(G, M, d)$ code}. Observe that in the case $G$ is the empty graph on $n$ vertices, a $(G, M, d)$ code is just an $(n, M, d)$ code, so our new notion of graphical error-correcting codes is a direct generalization of typical error-correcting codes. 

In this paper, we investigate the relationship between one-sided spectral expansion and this new sphere packing problem. Expander graphs play an important role in combinatorics and theoretical computer science. They also appear in the study of error-correcting codes, such as in the seminal work of Sipser and Spielman \cite{sipser1996}, who construct good error-correcting codes using expander graphs. By establishing a new connection between these subjects, we relate spectral expansion, an algebraic quantity, with the geometry of the space of proper $q$-colorings of a graph. 

\subsection{Setup}
Fix an integer $q \geq 3$. Let $G = (V, E)$ be a graph on $n$ vertices. A \emph{proper $q$-coloring} of $G$ is a map $X: V \to [q]$ such that $X(u) \neq X(v)$ whenever $uv \in E$. For the rest of the paper, unless otherwise specified, we will simply refer to a proper $q$-coloring as a coloring. Identifying $V$ with $[n]$, we can view $X$ as an element of $[q]^n$. For two colorings $X, Y$ of $G$, their \emph{Hamming distance} is given by
\[
	d_H(X, Y) = |\{i \in [n] | X(i) \neq Y(i)\}|.
\]
We say that a set of colorings of $G$ is \emph{$\delta$-distinct} if any pair of colorings in the set have Hamming distance at least $\delta n$. In other words, a $\delta$-distinct set of colorings is a $(G, M, \delta n)$ code. 

We also introduce another notion of distance, which seems more natural in the context of graph colorings. For two colorings $X, Y$ of $G$, their \emph{coloring distance} is given by 
\[
	d_c(X, Y) = \min_{\sigma \in S_n} d_H(X, \sigma(Y)).
\]
This is simply the minimum Hamming distance between $X$ and $\sigma(Y)$. Similarly, we say that a set of colorings of $G$ is \emph{strongly $\delta$-distinct} if any pair of colorings in the set have coloring distance at least $\delta n$. Clearly, $d_c(X, Y) \leq d_H(X, Y)$, so any strongly $\delta$-distinct set of colorings is also $\delta$-distinct. While there is no obvious relationship between these two distances beside the trivial inequality, most of our asymptotic results will hold under either notion of distance. 

Throughout the rest of this paper, we only work with regular graphs unless otherwise specified. Given a $d$-regular graph $G$, let $A$ denote its normalized adjacency matrix, with entries given by 
\[
	A_{uv} = 
	\begin{cases}
		\frac{1}{d} \quad &\text{if } uv \in E\\
		0 \quad &\text{otherwise}.
	\end{cases}
\]
Let $\lambda_1 \geq \lambda_2 \geq \cdots \geq \lambda_n$ denote its eigenvalues. Recall that $\lambda_1 = 1$ and $\lambda_n \geq -1$. 

For any $0 \leq \delta \leq 1$, let $\mathcal{C}_{G, \delta}$ and $\mathcal{S}_{G, \delta}$ denote the (finite) collections of $\delta$-distinct and strongly $\delta$-distinct sets of $q$-colorings of $G$, respectively. For positive integers $n$ and $d$ and real number $-1 \leq \lambda \leq 1$, let $\mathcal{G}_{n, d, \lambda}$ denote the (finite) set of $d$-regular graphs on at most $n$ vertices with $\lambda_2 \leq \lambda$. For convenience and generality, we allow for the empty graph on $n$ vertices, which we regard as a $0$-regular graph with $\lambda_2 = 1$. 

For any $(\delta, \lambda) \in [0, 1] \times [-1, 1]$ and $d \geq 0$, we define
\[
	f_{\delta, \lambda, d}(n) = \max_{G \in \mathcal{G}_{n, d, \lambda}} \max_{C \in \mathcal{C}_{G, \delta}} |C|
\]
and 
\[
	g_{\delta, \lambda, d}(n) = \max_{G \in \mathcal{G}_{n, d, \lambda}} \max_{C \in \mathcal{S}_{G, \delta}} |C|,
\]
where the maximum outputs $0$ if it is taken over an empty set. Observe that $f_{\delta, \lambda, 0}(n)$ is precisely the largest $M$ such that there exists an $(n, M, \delta n)$ code. Also, we trivially have that $f_{\delta, \lambda, d}(n) \geq g_{\delta, \lambda, d}(n)$. 

The main subject of study of this paper is the asymptotic behavior of $f$ and $g$ when we fix $\delta$ and $\lambda$. By the Plotkin bound \cite{vanlint1999}, the largest $(n, M, \delta n)$ code when $\delta > 1 - 1/q$ has constant size, which implies that $f$ is bounded by a constant when $\delta > 1 - 1/q$. We also have that $d_c(X, Y) \leq (1 - 1/q)n$ for any pair of colorings, which trivializes $g$ whenever $\delta > 1 - 1/q$. Thus, we will restrict our attention to $\delta \in [0, 1 - 1/q]$ for the rest of the paper. Furthermore, the Alon-Boppana bound \cite{nilli1991} states that $d$-regular graphs have $\lambda_2 \geq 2\sqrt{d - 1}/d+ o(1) > 0$. Therefore, we will also make the restriction that $\lambda \in (0, 1]$.

\begin{remark}\label{rmk:one-sided}
We focus on one-sided expanders (graphs with $\lambda_2 \leq \lambda$) instead of two-sided expanders (graphs with $\max \{|\lambda_2|, |\lambda_n|\} \leq \lambda$). This is due to Hoffman's bound \cite{hoffman1970} on the chromatic number of a two-sided expander:
\[
    \chi(G) \geq 1 - \frac{1}{\lambda_n} \geq 1 + \frac{1}{\lambda}.
\]
In particular, this implies that a two-sided expander with $\lambda < \frac{1}{q - 1}$ cannot be $q$-colorable. Nevertheless, we leave as a future direction to investigate the problem for two-sided expanders above this bound.
\end{remark}

Next, we define the relevant asymptotic regimes of $f$ and $g$ which we study. We say that $(\delta, \lambda)$ is in the \emph{exponential regime} for $f$ (resp. $g$), denoted $R_{f, exp}$ (resp. $R_{g, exp}$), if
\[
	\sup_{d \geq 0} \limsup_{n \to \infty} \frac{\log f_{\delta, \lambda, d}(n)}{n} > 0.
\]
We say that $(\delta, \lambda)$ is in the \emph{constant regime} for $f$ (resp. $g$), denoted $R_{f, con}$ (resp. $R_{g, con}$), if
\[
	\sup_{d \geq 0} \limsup_{n \to \infty} f_{\delta, \lambda, d}(n) < \infty.
\]
Finally, we say that $(\delta, \lambda)$ is in the \emph{unique regime} for $g$, denoted $R_{g, uni}$, if
\[
	\sup_{d \geq 0} \limsup_{n \to \infty} g_{\delta, \lambda, d}(n) = 1.
\]
Observe that the unique regime only applies to the coloring distance because for the weaker Hamming distance, we can always cyclically permute the colors of one coloring to get another coloring at Hamming distance $n$. 

A simple yet useful observation is the monotonicity of $f$ and $g$. 
\begin{lemma}\label{lem_monotone}
	For fixed $n$ and $d$, the functions $f_{\delta, \lambda, d}(n)$ and $g_{\delta, \lambda, d}(n)$ are weakly decreasing in $\delta$ and weakly increasing in $\lambda$. 
\end{lemma}
As a corollary of \Cref{lem_monotone}, if $(\delta, \lambda)$ is in the exponential regime for $g$, then any $(\delta', \lambda')$ with $\delta' \leq \delta$ and $\lambda' \geq \lambda$ is in the exponential regime for both $f$ and $g$. Similarly, if $(\delta, \lambda)$ is in the constant regime for $f$, then any $(\delta', \lambda')$ with $\delta' \geq \delta$ and $\lambda' \leq \lambda$ is in the constant regime for both $f$ and $g$. Finally, if $(\delta, \lambda)$ is in the unique regime (for $g$), then any $(\delta', \lambda')$ with $\delta' \geq \delta$ and $\lambda' \leq \lambda$ is in the unique regime.

\subsection{Main Results}
    The main results of this paper establish a landscape for the different growth regimes for the functions $f$ and $g$ (see \Cref{fig_regimes}). For every main theorem except for \Cref{thm_unique}, which concerns the unique regime, we prove it using the notion of distance that implies the stronger statement. That is, we prove results about the exponential regime using $g$ and results about the constant regime using $f$. 
    
    First, we construct an infinite family of one-sided expander graphs that have an exponential number of strongly $\delta$-distinct colorings for any $\delta < 1 - 1/q$, which is the largest possible. 

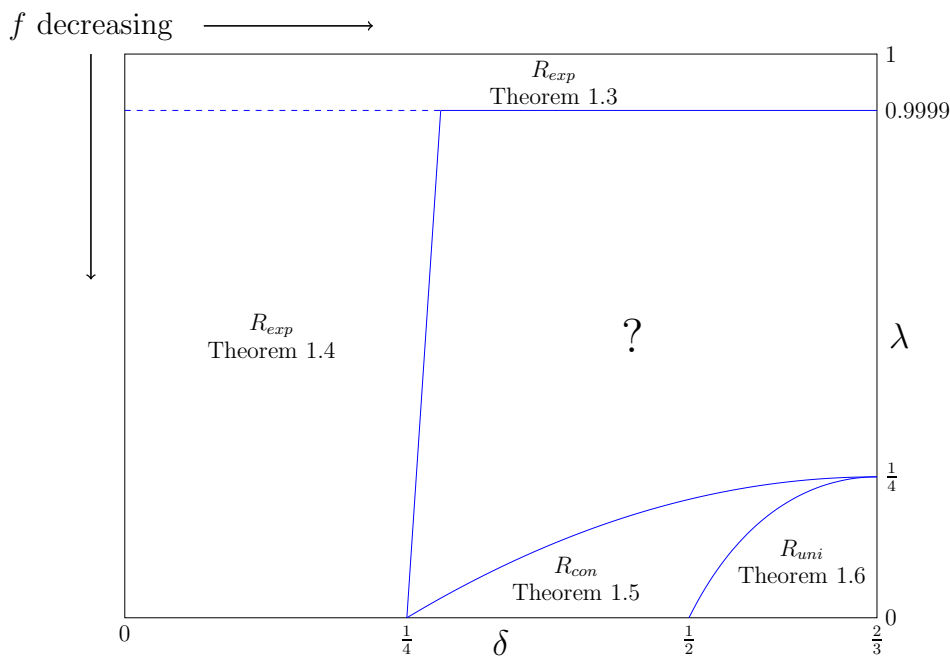
\begin{figure}[ht]
    \centering
    \scalebox{0.8}{
    \begin{tikzpicture}[scale=10, xscale = 2]
      \draw (0,0) rectangle (2/3,1);
    
      \draw (1/3, 0) node[below, scale = 1.5] {$\delta$};
      \draw (2/3, 0.5) node[right, scale = 1.5] {$\lambda$};
      \draw (0, 0) node[below] {$0$};
      \draw (2/3, 0) node[below] {$\frac{2}{3}$};
      \draw (1/4, 0) node[below] {$\frac{1}{4}$};
      \draw (1/2, 0) node[below] {$\frac{1}{2}$};
      \draw (2/3, 1) node[right] {$1$};
      \draw (2/3, 0) node[right] {$0$};
      \draw (2/3, 1/4) node[right] {$\frac{1}{4}$};
      \draw (2/3, 0.9) node[right] {$0.9999$};
      
      \draw[domain=1/2:2/3, smooth, variable=\x, blue] plot ({\x}, {1 + 1/(12 * \x * \x - 16 * \x + 4)});
      
      \draw[domain=1/4:2/3, smooth, variable=\x, blue] plot ({\x}, {1/4 - 36/25 * (\x - 2/3) * (\x - 2/3)});
    
      \draw [-, blue] (1/4, 0) -- (0.28, 0.9);
      \draw [-, blue] (0.28, 0.9) -- (2/3, 0.9);
      \draw [-, blue, dashed] (0, 0.9) -- (0.28, 0.9);
      
      \draw [->, thick] (-0.03, 1) -- (-0.03, 0.6);
      \draw [->, thick] (0.07, 1.05) -- (0.22, 1.05);
      
      \draw (-0.03,1.05) node[scale = 1.3] {$f$, $g$ decreasing};
      \draw (0.6,0.1) node[align=center] {$R_{g, uni}$ \\ \Cref{thm_unique}};
      \draw (0.13,0.5) node[align=center] {$R_{g, exp}$ \\ \Cref{thm_bipartite_exp}};
      \draw (0.38,0.95) node[align=center] {$R_{g, exp}$ \\ \Cref{thm_exp}};
      \draw (0.4,0.07) node[align=center] {$R_{f, con}$ \\ \Cref{thm_const}};
      \draw (0.45,0.5) node[align=center, scale = 2] {?};
    \end{tikzpicture}
    }
    \caption{Schematic diagram for the three regimes when $q = 3$. The curves bounding the regimes may not be precise. Note that the boundary of the exponential regime at $(1/4, 0)$ has a positive slope. The growth of $f$ and $g$ in the central region is unknown. We label the regimes with the function that implies the stronger statement.}
    \label{fig_regimes}
\end{figure}

\begin{theorem}\label{thm_exp}
	There exists some universal constant $\alpha > 0$ independent of $q$ such that any $(\delta, 1 - \alpha)$ with $\delta < 1 - 1/q$ is in the exponential regime for both $f$ and $g$. 
\end{theorem}
    We also show that bipartite expanders with arbitrarily small $\lambda_2$ can have an exponential number of strongly $\delta$-distinct colorings for a certain positive value of $\delta$.
\begin{theorem}\label{thm_bipartite_exp}
	For any $\lambda > 0$, there exists $\epsilon = \epsilon(\lambda) > 0$ such that $\left(1 - \frac{1}{2} \left( \frac{1}{\lfloor \frac{q}{2} \rfloor} + \frac{1}{\lceil \frac{q}{2} \rceil} \right) + \epsilon, \lambda \right)$ is in the exponential regime for both $f$ and $g$.
\end{theorem}

    Conversely, we show a sharp phase transition between the exponential regime and the constant regime at the threshold above for $\lambda \to 0$, for both $f$ and $g$. In other words, for graphs with nearly perfect spectral expansion, increasing $\delta$ even slightly above the threshold in \Cref{thm_bipartite_exp} bounds the number of possible $\delta$-distinct colorings by a constant. Furthermore, we establish a sharp boundary point for the constant regime for $f$ and $g$ on the line $\delta = 1 - 1/q$. 
\begin{theorem}\label{thm_const}
    The following hold:
    \begin{enumerate}
        \item For any $\epsilon > 0$, there exists some $\lambda = \lambda(\epsilon) > 0$ such that $\left(1 - \frac{1}{2} \left( \frac{1}{\lfloor \frac{q}{2} \rfloor} + \frac{1}{\lceil \frac{q}{2} \rceil} \right) + \epsilon, \lambda \right)$ is in the constant regime for both $f$ and $g$.
        \item For any $\epsilon > 0$, $\left(1 - \frac{1}{q}, \frac{1}{(q - 1)^2} - \epsilon \right)$ is in the constant regime for both $f$ and $g$. Conversely, $\left(1 - \frac{1}{q}, \frac{1}{(q - 1)^2}\right)$ is not in the constant regime for either $f$ or $g$.
    \end{enumerate}
\end{theorem}

    Finally, we establish sharp boundary points for the unique regime (for $g$). In particular, the boundary point on the line $\delta = 1 - 1/q$ coincides with that of the constant regime for $f$.
\begin{theorem}\label{thm_unique}
    The following hold: 
    \begin{enumerate}
        \item For any $\epsilon > 0$, there exists some $\lambda = \lambda(\epsilon) > 0$ such that $\left(1 - \frac{1}{q - 1} + \epsilon, \lambda \right)$ is in the unique regime. Conversely, $\left(1 - \frac{1}{q - 1}, \lambda \right)$ is not in the unique regime for all $\lambda > 0$.
        \item For any $\epsilon > 0$, $\left(1 - \frac{1}{q}, \frac{1}{(q - 1)^2} - \epsilon \right)$ is in the unique regime. 
    \end{enumerate}
\end{theorem}

\subsection{Outline}
The rest of the paper is organized as follows. In \Cref{sec_prelim}, we recall some preliminary results and prove \Cref{lem_monotone}. In \Cref{sec_exp}, we study the exponential regime and prove \Cref{thm_exp,thm_bipartite_exp}. In \Cref{sec_const_regime}, we study the constant regime and prove \Cref{thm_const}. In \Cref{sec_unique_regime}, we study the unique regime and prove \Cref{thm_unique}. In \Cref{sec_future}, we discuss some possible future directions. 

\section*{Acknowledgments}
This work was carried out as part of a direct funding project in the Undergraduate Research Opportunities Program at MIT. The author thanks Yufei Zhao for suggesting the problem and for his mentorship in this project. We also thank Mitali Bafna for sharing her research which inspired this work. We thank Mitali Bafna and Yufei Zhao for helpful discussions. We thank Sammy Luo and Yufei Zhao for valuable comments which improved this paper. Finally, we thank the MIT UROP office and the MIT Department of Mathematics for providing this research opportunity.

\section{Preliminaries}\label{sec_prelim}
In this section, we recall some preliminary results about expander graphs and tools that we apply throughout the rest of the paper. 

We first prove \Cref{lem_monotone}.
\begin{proof}[Proof of \Cref{lem_monotone}]
    Fixing $n, d, \lambda$, we see that $f$ and $g$ weakly decrease in $\delta$ as the collections $\mathcal{C}_{G, \delta}$ and $\mathcal{S}_{G, \delta}$ shrink as $\delta$ increases. Similarly, fixing $n, d, \delta$, we see that $f$ and $g$ weakly increase in $\lambda$ as the collection $G_{n, d, \lambda}$ grows as $\lambda$ increases. 
\end{proof}

Next, we recall several general preliminary results we use in later sections. First, recall the Courant-Fischer characterization of the second eigenvalue. 
\begin{lemma}\label{lem_courant}
	Let $G = (V, E)$ be a regular graph and let $A$ be its normalized adjacency matrix. Then the second eigenvalue $\lambda_2$ of $A$ satisfies
	\[
		\lambda_2 = \max_{\substack{x \neq 0 \\ x \perp \mathbf{1}}} \frac{\langle x, Ax \rangle}{\langle x, x \rangle}.
	\]
\end{lemma}

Also recall Cheeger's inequality, which relates spectral expansion with edge expansion. 
\begin{lemma}[\cite{cheeger1971}]\label{lem_cheeger}
    Let $G = (V, E)$ be a $d$-regular graph and let 
    \[
        h = h(G) = \min_{\substack{S \subset V \\ 0 < |S| \leq |V|/2}} \frac{|E(S, V \setminus S)|}{|S|}
    \]
    be its edge expansion ratio. Then
    \[
        \frac{d(1 - \lambda_2)}{2} \leq h \leq d\sqrt{2(1 - \lambda_2)}.
    \]
\end{lemma}

A main strategy we use in constructions in \Cref{sec_exp} is random coloring. For this purpose, we recall McDiarmid's inequality, also known as the bounded difference inequality, which provides a concentration bound on Lipschitz functions of independent random inputs. 
\begin{lemma}[\cite{mcdiarmid1989}]\label{lem_bounded_diff}
    Let $X_1 \in \Omega_1, \ldots, X_n \in \Omega_n$ be independent random variables. Suppose $f: \Omega_1 \times \cdots \times \Omega_n \to \R$ is $c$-Lipschitz; that is,
    \[
        |f(x_1, \ldots, x_n) - f(x'_1, \ldots, x'_n)| \leq c
    \]
    whenever $(x_1, \ldots, x_n)$ and $(x'_1, \ldots, x'_n)$ only differ on one coordinate. Then for any $t \geq 0$,
    \[
        \Pr[f(X_1, \ldots, X_n) \leq \E[f(X_1, \ldots, X_n)] - t] \leq e^{\frac{-2t^2}{c^2 n}}
    \]
    and 
    \[
        \Pr[f(X_1, \ldots, X_n) \geq \E[f(X_1, \ldots, X_n)] + t] \leq e^{\frac{-2t^2}{c^2 n}}
    \]
\end{lemma}

The following two results provide us with bipartite $d$-regular graphs that have ``optimal expansion''. However, for our purposes, we only need graphs with sufficiently good expansion. We use these results mainly for concreteness in the computation of implicit constants in our proofs. 
\begin{lemma}[\cite{brito2022}]\label{lem_eigenval_of_random_bipartite}
    Let $G$ be a uniformly random $d$-regular bipartite graph on $2n$ vertices. Then there exist a sequence of $\epsilon_n \to 0$ such that
    \[
        \lambda_2 \leq 2 \sqrt{d - 1}/d + \epsilon_n
    \]
    asymptotically almost surely as $n \to \infty$. 
\end{lemma}
The celebrated Alon-Boppana bound \cite{nilli1991} states that asymptotically $\lambda_2 \geq 2 \sqrt{d - 1}/d + o(1)$. Thus, the above result says that almost all bipartite $d$-regular graphs approach this bound. A similar result for random $d$-regular graphs is proved by Friedman in \cite{friedman2008}. 

Another seminal result by Marcus, Spielman, and Srivastava \cite{marcus2015} establishes the existence of $d$-regular bipartite graphs that achieve the Alon-Boppana bound, known as bipartite Ramanujan graphs.
\begin{lemma}[\cite{marcus2015}]\label{lem_ramanujan}
    For each $d \geq 3$, there exists an infinite sequence of $d$-regular bipartite graphs with
    \[
        \lambda_2 \leq 2 \sqrt{d - 1} / d.
    \]
\end{lemma}

\section{The exponential regime}\label{sec_exp}
In this section, we give constructions that prove \Cref{thm_exp,thm_bipartite_exp}. 

\subsection{Construction for \Cref{thm_exp}}
Fix $\delta < 1 - 1/q$ and let $\epsilon = 1 - 1/q - \delta$. Let $\alpha = 10^{-4}$. To prove \Cref{thm_exp}, we construct an infinite sequence of $3$-regular graphs $G_n$ with $\lambda_2(G_n) \leq 1 - \alpha$ such that for sufficiently large $n$, each graph $G_n$ has a strongly $\delta$-distinct set of $q$-colorings of size at least $e^{c|V(G_n)|}$, for some constant $c = c(\epsilon)$ to be specified. Such a construction shows that $(\delta, 1 - \alpha) \in R_{exp}$. 

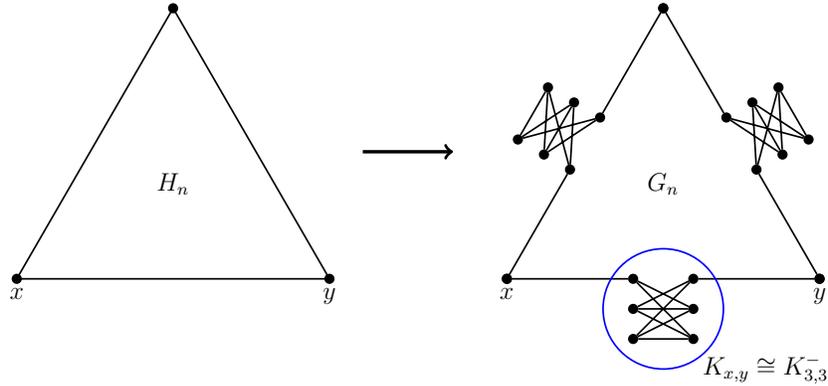
\begin{figure}[ht]
    \centering
    \scalebox{0.9}{
    \begin{tikzpicture}[thick, scale = 1]
        \tikzstyle{vertex}=[circle, draw, fill=black,inner sep=0pt, minimum width=4pt]
        \draw \foreach \x in {90, 210, 330} {
            (\x:3) node[vertex]{} -- (\x + 120: 3)
        };
        \node[below] at (210:3) {$x$};
        \node[below] at (330:3) {$y$};
        \node[below] at (0, 0.4) {$H_n$};
        \phantom{
        \draw \foreach \x in {30, 150, 270}{
            (\x - 60:3) node[vertex](x){}
            (\x + 60:3) node[vertex](y){}
            (\x: 1.5)
            ++(\x - 90:0.5) node[vertex](x1){}
            ++(\x: 0.5) node[vertex](x2){}
            ++(\x: 0.5) node[vertex](x3){}
            (\x: 1.5)
            ++(\x + 90:0.5) node[vertex](y1){}
            ++(\x: 0.5) node[vertex](y2){}
            ++(\x: 0.5) node[vertex](y3){}
            (x) -- (x1)
            (y) -- (y1)
            (x1) -- (y2)
            (x1) -- (y3)
            (x2) -- (y1)
            (x2) -- (y2)
            (x2) -- (y3)
            (x3) -- (y1)
            (x3) -- (y2)
            (x3) -- (y3)
        };
        \draw[blue] (0,-2) circle [radius=1];
        \node at (1.7,-3) {$K_{xy} \cong K_{3, 3}^-$};
        \draw (0, 0) -- (5, 0);
        }
    \end{tikzpicture}
    \begin{tikzpicture}[overlay]
          \draw[->,ultra thick] (-2,4) -- (-0.5,4);
    \end{tikzpicture}
    \begin{tikzpicture}[thick, scale = 1]
        \tikzstyle{vertex}=[circle, draw, fill=black,inner sep=0pt, minimum width=4pt]
        \draw \foreach \x in {30, 150, 270}{
            (\x - 60:3) node[vertex](x){}
            (\x + 60:3) node[vertex](y){}
            (\x: 1.5)
            ++(\x - 90:0.5) node[vertex](x1){}
            ++(\x: 0.5) node[vertex](x2){}
            ++(\x: 0.5) node[vertex](x3){}
            (\x: 1.5)
            ++(\x + 90:0.5) node[vertex](y1){}
            ++(\x: 0.5) node[vertex](y2){}
            ++(\x: 0.5) node[vertex](y3){}
            (x) -- (x1)
            (y) -- (y1)
            (x1) -- (y2)
            (x1) -- (y3)
            (x2) -- (y1)
            (x2) -- (y2)
            (x2) -- (y3)
            (x3) -- (y1)
            (x3) -- (y2)
            (x3) -- (y3)
        };
        \draw[blue] (0,-2) circle [radius=1];
        \node at (1.7,-3) {$K_{xy} \cong K_{3, 3}^-$};
        \node[below] at (210:3) {$x$};
        \node[below] at (330:3) {$y$};
        \node[below] at (0, 0.4) {$G_n$};
    \end{tikzpicture}
    }
    \caption{Illustration of the construction of $G_n$ from $H_n$.}
    \label{fig_exp_construction}
\end{figure}

First, let $H_n$ be a sequence of $3$-regular expander graphs with $\lambda_2(H_n) \leq 2\sqrt{2}/3$ given by \Cref{lem_ramanujan}. Let $K_{3, 3}^-$ denote the complete bipartite graph $K_{3, 3}$ with one edge removed and let $u, v$ be the two endpoints of the removed edge. We construct $G_n$ as follows. Start with the graph $H_n$. For each edge $xy \in E(H_n)$, we add a copy of $K_{xy} \cong K_{3, 3}^-$ to the vertex set of $G_n$. Remove the edge $xy$ and add edges $xu$ and $yv$. In other words, we replace each edge of $H_n$ with a $K_{3, 3}^-$ ``gadget'' that is connected to the two endpoints of the edge (see \Cref{fig_exp_construction}). It is easy to see that $G_n$ is still $3$-regular. For each edge of $H_n$, we added $6$ vertices. Thus, $G_n$ has $10 \cdot |V(H_n)|$ vertices. Write $H = H_n$, $G = G_n$, and let $N = |V(H_n)|$. We show that $G$ has the desired properties.

\begin{lemma}\label{lem_exp_construction_eigenval}
	Let $G = G_n$ be the graph defined as above. Then $\lambda_2 \leq 1 - \alpha$.
\end{lemma}
\begin{proof}
    Let $S \subset V$ be any vertex subset with $|S| \leq N/2$. We will show that 
	\[
		\frac{|E(S, V \setminus S)|}{|S|} \geq 3 \sqrt{2 \alpha}.
	\]
    By \Cref{lem_cheeger}, this then implies $\lambda_2 \leq 1 - \alpha$.
	
	For each gadget $K_{xy}$, if $0 < |S \cap V(K_{xy})| < 6$, then this gadget contributes at least $2$ edges to $E(S, V \setminus S)$, whereas it contributes at most $5$ vertices to $|S|$. Since, by our choice of $\alpha$, we have $3 \sqrt{2 \alpha} < 2/5$, $S$ cannot be a union of such ``incomplete gadgets.'' Now let $S'$ be obtained from $S$ by deleting vertices of $S$ in each incomplete gadget. Then $S'$ is nonempty and each incomplete gadget reduces $E(S', V \setminus S')$ by at least $1$ and $|S'|$ by at most $5$. If $k$ is the number of incomplete gadgets in $S$, then
	\[
	\frac{|E(S', V \setminus S')|}{|S'|} \leq \frac{|E(S, V \setminus S)| - k}{|S| - 5k} \leq \frac{|E(S, V \setminus S)|}{|S|}
	\]
	since $3 \sqrt{2 \alpha} < 1/5$. Thus, we may assume that $S$ has no incomplete gadgets. 
	
	Let $k_1$ be the number of gadgets $K_{xy} \subset S$ such that $x, y \in S$. Let $k_2$ be the number of gadgets $K_{xy} \subset S$ such that exactly one of $x, y$ is in $S$. Let $k_3$ be the number of gadgets $K_{xy} \subset S$ such that $x, y \notin S$. Then letting $S_H = S \cap V(H)$,
	\[
		\frac{|E(S, V \setminus S)|}{|S|} = \frac{3|S_H| - 2k_1 + 2k_3}{|S_H| + 6k_1 + 6k_2 + 6 k_3}.
	\]
    If the left hand side is at least $1/3$, we are already done. Otherwise, removing the $k_3$ gadgets decreases the edge expansion of $S$. Thus, we can assume $k_3 = 0$.
	
	Suppose $|S_H| \geq 3N/5$. Since $|S| = |S_H| + 6k_1 + 6k_2 \leq |V|/2 = 5N$, we have that $2k_1 \leq 2k_1 + 2k_2 \leq \frac{5}{3} N - \frac{1}{3}|S_H|$. Thus, 
	\[
		\frac{|E(S, V \setminus S)|}{|S|} \geq \frac{\frac{10}{3}|S_H| - \frac{5}{3} N}{5 N} \geq \frac{1}{15} \geq 3 \sqrt{2 \alpha}.
	\]
	
	Thus we must have $|S_H| \leq 3N/5$. Let $\overline{S_H} = V_H \setminus S_H$. Observe that $k_1$ is exactly equal to the number of edges of $H$ between vertices of $S_H$ when viewing it as a subset of $V(H)$; that is, $k_1 = E_H(S_H)$. Similarly, $k_2 = E_H(S_H, \overline{S_H})$. Since $H$ is $3$-regular, $3|S_H| = 2 E_{H}(S_H) + E_H(S_H, \overline{S_H})$. Thus, 
	\[
		3 |S_H| = 2 k_1 + k_2. 
	\]
	Then
	\begin{align*}
		\frac{|E(S, V \setminus S)|}{|S|}
		&\geq \frac{3|S_H| - 2k_1}{|S_H| + 12k_1 + 6k_2}\\
		&\geq \frac{E_{H}(S_H, \overline{S_H})}{19|S_H|}\\
		&\geq \frac{h(H)}{19} \min \left\{1, \frac{|\overline{S_H}|}{|S_H|} \right\}\\
		&\geq \frac{2h(H)}{57}
	\end{align*}
	By our choice of $H$ and \Cref{lem_cheeger}, this is at least $\geq \frac{2\sqrt{2}}{57} \geq 3 \sqrt{2 \alpha}$. Thus, $h(G) \geq 3 \sqrt{2 \alpha}$ and \Cref{lem_cheeger} implies that $\lambda_2 \leq 1 - \alpha$. 
\end{proof}

\begin{lemma}\label{lem_exp_construction_colorings}
	Let $G = G_n$ be the graph defined as above with $|V(G)| = 10|V(H)| = 10N$. Then for $n$ sufficiently large, $G$ has a strongly $\delta$-distinct set of $q$-colorings of size at least $e^{cN}$, where $c = \epsilon^2/8$ and $\epsilon = 1 - 1/q - \delta$. 
\end{lemma}
\begin{proof}
	Recall that $V(G) = V(H) \cup \bigcup_{xy \in E(H)}V(K_{xy})$, where the gadget $K_{xy}$ is a copy of $K_{3, 3}^{-}$. Consider the following random $q$-coloring of $G$. First color each vertex in $V(H)$ independently and uniformly at random. For each $xy \in E(H)$, we have two cases. If $x$ and $y$ are assigned the same color, say $i$, then color one part of $K_{xy}$ by $i + 1 \mod q$ and the other part by $i + 2 \mod q$. If $x$ and $y$ are assigned different colors $i$ and $j$ respectively, then color the part of $K_{xy}$ with an edge to $x$ by $j$ and the other part by $i$. 

    Observe that under this random coloring, the marginal distribution of the color of any vertex is uniform. Thus, for $X, Y$ independent samples of this random coloring, we have
    \[
        \E[d_c(X, Y)] = \min_{\sigma \in S_q}\E[d_H(X, \sigma(Y))] = \left(1 - \frac{1}{q}\right) |V(G)|.
    \]
    For any $\sigma \in S_q$, let $x_1, \ldots, x_N$ be the colors of the vertices in $H$ under $X$ and let $y_1, \ldots, y_N$ be the colors of the vertices in $H$ under $\sigma(Y)$. Since the coloring is uniquely determined by the colors of vertices in $H$, we can write
    \[
        d_H(X, \sigma(Y)) = f_{\sigma}(x_1, \ldots, x_N, y_1, \ldots, y_N). 
    \]
    Observe that changing one input of $f$ changes the colors of at most $1 + 3 \cdot 6 = 19$ vertices; namely, one vertex in $H$ and the vertices of the three gadgets adjacent to it. Thus, $f_{\sigma}$ is $19$-Lipschitz. By \Cref{lem_bounded_diff}, 
    \[
        \Pr[d_H(X, \sigma(Y)) \leq \delta |V(G)|]
        \leq \exp{\left( \frac{-2 (10 \epsilon N)^2}{19^2 \cdot 2N} \right)}
        \leq \exp{\left( \frac{-\epsilon^2 N}{3.9} \right)}.
    \]
    By union bound, we have that 
    \[
        \Pr[d_c(X, Y) \leq \delta |V(G)|] \leq q! \exp{\left( \frac{-\epsilon^2 N}{3.9} \right)}.
    \]
    By union bound again, for a collection $C$ of $e^{cN}$ i.i.d. samples of the random colorings where $c = \epsilon^2/8$,
    \[
        \Pr[C \text{ is strongly $\delta$-distinct}] \geq 1 - |C|^2 q! \exp{\left( \frac{-\epsilon^2 N}{3.9} \right)} > 0
    \]
    for $N$ sufficiently large. 
\end{proof}

\begin{proof}[Proof of \Cref{thm_exp}]
	\Cref{lem_exp_construction_eigenval,lem_exp_construction_colorings} together show that $(\delta, 1 - \alpha) \in R_{g, exp}$ for $\alpha = 10^{-4}$. 
\end{proof}

\subsection{Construction for \Cref{thm_bipartite_exp}}
For any $\lambda > 0$, choose $d$ with $2/\sqrt{d} \leq \lambda$. Let $G = G_n$ be a sequence of uniformly random bipartite $d$-regular graphs on $2n$ vertices. By \Cref{lem_eigenval_of_random_bipartite}, with probability $1 - o(1)$, each $G_n$ has $\lambda_2 \leq \lambda$. Let $\delta = 1 - \frac{1}{2} \left( \frac{1}{\lfloor \frac{q}{2} \rfloor} + \frac{1}{\lceil \frac{q}{2} \rceil} \right)$. We show that some random coloring scheme of $G$ produces colorings at distance exceeding $\delta$ with high probability.
\begin{lemma}\label{lem_d_XY}
    Let $G$ be a uniformly random $d$-regular bipartite graph on $2n$ vertices. Then, there exists a random coloring scheme of $G$ such that for any pair of independent samples of this coloring, $X$ and $Y$, there exist constants $\tau = \tau(d, q) > 0$ and $c = c(d, q) > 0$ such that 
    \[
        \Pr \left[d_c(X, Y) < \left( \delta + \tau^2 \right) 2n \right] \leq e^{-cn}
    \]
    for sufficiently large $n$.
\end{lemma}
\begin{proof}
    We define the random coloring as follows. Let $\tau = \tau(d, q) > 0$ be a parameter to be determined. Let the two parts of $G$ be denoted by $V_1$ and $V_2$. Let $C_1 = \{1, \ldots, \lfloor q/2\rfloor\}$ and $C_2 = [q]\setminus C_1$. Start with a coloring of $G$ where we color vertices of $V_1$ by colors in $C_1$ uniformly and independently at random, and color the vertices of $V_2$ by colors in $C_2$ uniformly and independently at random. This is clearly a proper coloring. Now, for each vertex of $V_1$, independently replace its color by $q$ with probability $\tau$, and let $Q \subset V_1$ denote the set of vertices recolored this way. The coloring may now be invalid because some neighbors of $Q$ could also have color $q$. We then recolor all neighbors of $Q$, $N(Q) \subset V_2$, by the color $q - 1$. Since $q - 1 \notin C_1$, we now have a proper coloring. 

    We first give the intuition behind this coloring scheme. Before doing the recoloring, two such independent random colorings have expected Hamming distance exactly equal to $\delta (2n)$. In order to make the linear gain, we color a small linear fraction of $V_1$ by $q$. On expectation, the collision between the recolored vertices of the two colorings is very small. On the other hand, the remaining recolored vertices are guaranteed to have a different color under the two colorings, producing the excess we need. 

    Let $X$ and $Y$ be two independent samples of this random coloring. Let $Q_X, Q_Y \subset V_1$ denote the sets of vertices recolored $q$ by $X$ and $Y$ respectively. Let $\sigma \in S_q$. We break down $d_H(X, \sigma(Y))$ and consider two cases: $\sigma(C_1) = C_1$ or $\sigma(C_1) \neq C_1$. 
    \begin{enumerate}
        \item Suppose $v \in V_1 \setminus (Q_X \cup Q_Y)$. 
        \begin{enumerate}
            \item If $\sigma(C_1) = C_1$, then $\Pr[X(v) = \sigma(Y(v))] = \frac{1}{|C_1|}$. On expectation, the Hamming distance that these vertices contribute is $(n - |Q_X \cup Q_Y|) (1 - 1/|C_1|)$.
            \item If $\sigma(C_1) \neq C_1$, then $\Pr[X(v) = \sigma(Y(v))] \leq 1/|C_1| - 1/|C_1|^2$. On expectation, the Hamming distance that these vertices contribute is $(n - |Q_X \cup Q_Y|) (1 - 1/|C_1| + 1/|C_1|^2)$.
        \end{enumerate}
        \item Suppose $v \in V_2 \setminus (N(Q_X) \cup N(Q_Y))$.
        \begin{enumerate}
            \item If $\sigma(C_1) = C_1$, then $\Pr[X(v) = \sigma(Y(v))] = \frac{1}{|C_2|}$. On expectation, the Hamming distance that these vertices contribute is $(n - |N(Q_X) \cup N(Q_Y)|) (1 - 1/|C_2|)$.
            \item If $\sigma(C_1) \neq C_1$, then $\Pr[X(v) = \sigma(Y(v))] \leq 1/|C_2| - 1/|C_2|^2$. On expectation, the Hamming distance that these vertices contribute is $(n - |N(Q_X) \cup N(Q_Y)|) (1 - 1/|C_2| + 1/|C_2|^2)$.
        \end{enumerate}
        \item Suppose $v \in Q_X \triangle Q_Y$. 
        \begin{enumerate}
            \item If $\sigma(C_1) = C_1$, then $\Pr[X(v) = \sigma(Y(v))] = 0$. The Hamming distance that these vertices contribute is $|Q_X \triangle Q_Y|$. 
        \end{enumerate}
        \item Suppose $v \in N(Q_X) \triangle N(Q_Y)$. 
        \begin{enumerate}
            \item If $\sigma(C_1) = C_1$, then $\Pr[X(v) = \sigma(Y(v))] = \frac{1}{|C_2|}$. On expectation, the Hamming distance that these vertices contribute is $|N(Q_X) \triangle N(Q_Y)| (1 - 1/|C_2|)$. 
        \end{enumerate}
    \end{enumerate}
    When $\sigma(C_1) \neq C_1$, the first two cases already give us 
    \begin{align*}
        \E[d_H(X, \sigma(Y))]
        &\geq n \left(1 - \frac{1}{|C_1|} \right) + n \left(1 - \frac{1}{|C_2|} \right) + n \left( \frac{1}{|C_1|^2} + \frac{1}{|C_2|^2} \right) \\
        &\quad - |Q_X \cup Q_Y| - |N(Q_X) \cup N(Q_Y)|\\
        &\geq 2 \delta n + \frac{2}{q^2} n - |Q_X \cup Q_Y| - |N(Q_X) \cup N(Q_Y)|.
    \end{align*}
    Since $Q_X$ and $Q_Y$ are chosen randomly and since $G$ is a random $d$-regular bipartite graph, $|Q_X \cup Q_Y| \approx 2 \tau n$ and $|N(Q_X) \cup N(Q_Y)| \approx 2 d \tau n$. We can thus choose $\tau$ small enough such that by \Cref{lem_bounded_diff}, there is some constant $c_1 = c_1(\tau) > 0$ such that with probability at least $1 - e^{c_1 n}$, 
    \[
    d_H(X, \sigma(Y)) \geq 2 \delta n + d \tau n.
    \]

    When $\sigma(C_1) \neq C_1$, we have
    \begin{align*}
        &\E[d_H(X, \sigma(Y))]\\
        &\geq n \left(1 - \frac{1}{|C_1|} \right) + n \left(1 - \frac{1}{|C_2|} \right) - |Q_X \cup Q_Y| \left(1 - \frac{1}{|C_1|} \right) \\
        &\quad - |N(Q_X) \cup N(Q_Y)| \left(1 - \frac{1}{|C_2|} \right) + |Q_X \triangle Q_Y| + |N(Q_X) \triangle N(Q_Y)| \left(1 - \frac{1}{|C_2|} \right)\\
        &\geq 2 \delta n + \left(|Q_X \triangle Q_Y| - |Q_X \cup Q_Y| \left(1 - \frac{1}{|C_1|} \right) \right) - |N(Q_X) \cap N(Q_Y)|.
    \end{align*}
    Again, we have that $|Q_X \triangle Q_Y| \approx 2 \tau n - \tau^2 n$, $|Q_X \cup Q_Y| \approx 2 \tau n$ and $|N(Q_X) \cap N(Q_Y)| \approx d^2 \tau^2 n$. Thus, we can still choose $\tau$ such that by \Cref{lem_bounded_diff}, there is some constant $c_2 = c_2(\tau) > 0$ such that with probability at least $1 - e^{c_2 n}$, 
    \[
        d_H(X, \sigma(Y)) \geq 2 \delta n + 2 \tau^2 n. 
    \]
    
    Finally, by union bound, 
    \[
        \Pr \left[d_c(X, Y) < \left( \delta + \tau^2 \right) 2n \right] \leq q! e^{-c_2 n} \leq e^{-cn}.
    \]
    for sufficiently large $n$.
\end{proof}

Now we are in a position to prove \Cref{thm_bipartite_exp}. 
\begin{proof}[Proof of \Cref{thm_bipartite_exp}]
    Let $\tau = \tau(d, q)$ and $c = c(d, q)$ be parameters such that \Cref{lem_d_XY} holds. For a collection $C$ of i.i.d. samples of the random colorings of $G$ as in \Cref{lem_d_XY}, union bound gives
    \[
        \Pr[C \text{ is strongly $(\delta + \tau^2)$-distinct}] \geq 1 - |C|^2 e^{-cn}. 
    \]
    We also need $G$ to have $\lambda_2 \leq \lambda$. This happens with probability $1 - o(1)$ by \Cref{lem_eigenval_of_random_bipartite}. Therefore, for all sufficiently large $n$, we can find some bipartite $d$-regular graph $G$ on $2n$ vertices with $\lambda_2 \leq \lambda$ that has a $(\delta + \tau^2)$-distinct set of colorings of exponential size. 
\end{proof}

\section{The constant regime}\label{sec_const_regime}
\Cref{thm_bipartite_exp} shows that the exponential regime for $g$ (and therefore for $f$ as well) covers all points $(\delta, \lambda)$ satisfying $\delta \leq 1 - \frac{1}{2} \left( \frac{1}{\lfloor \frac{q}{2} \rfloor} + \frac{1}{\lceil \frac{q}{2} \rceil} \right) + \epsilon$ for some $\epsilon = \epsilon(\lambda)$. The threshold here is roughly achieved by bipartite graphs, coloring one part randomly using half of the colors and the other part randomly using the rest of the colors. We show in this section that this threshold is essentially tight, and furthermore that there is a sharp phase transition at $\lambda = 0$ from the exponential regime to the constant regime for both $f$ and $g$. However, \Cref{thm_bipartite_exp} and \Cref{thm_const}~(1) still leave a gap (which vanishes as $\lambda \to 0$) between the exponential regime and the constant regime and it is not known whether a different growth behavior occurs within this gap (see \Cref{fig_regimes}). 

In this section, we also prove \Cref{thm_const}~(2), which establishes a sharp boundary point for the constant regime for $f$ and $g$ on the line $\delta = 1 - 1/q$. Unlike the first part of the theorem, this part, along with \Cref{thm_exp}, still leaves a constant size gap between the exponential regime and the constant regime for both $f$ and $g$ on this line (see \Cref{fig_regimes}), which presents an interesting open problem.

\subsection{Proof of \Cref{thm_const}~(1)}\label{subsec_4.1}
We prove the following stronger result, which implies the first part of \Cref{thm_const}.
\begin{theorem}\label{thm_strong_const}
	For any $\epsilon > 0$, there exists some $\lambda = \lambda(\epsilon) > 0$ and some integer $K = K(\epsilon)$ such that for $\delta = 1 - \frac{1}{2} \left( \frac{1}{\lfloor \frac{q}{2} \rfloor} + \frac{1}{\lceil \frac{q}{2} \rceil} \right) + \epsilon$, we have
	\[
		\sup_{d \geq 0} \limsup_{n \to \infty} f_{\delta, \lambda, d}(n) \leq K.
	\]
\end{theorem}
In other words, for any gap $\epsilon > 0$, we can find $\lambda > 0$ such that regular graphs with $\lambda_2 \leq \lambda$ exhibit a ``jump'' from having exponentially many strongly $\delta$-distinct colorings to having at most $K$ $(\delta + \epsilon)$-distinct colorings. 

One crucial approach we use in this section is grouping vertices by the tuple of colors they receive by various colorings and treating vertices in the same group as a single ``weighted'' vertex, connected by ``weighted'' edges to other weighted vertices. As such, it is convenient to suppress the size of the graph and measure vertex or edge subsets by the fraction of vertices/edges they contain. In particular, this makes it conceptually more intuitive when we disregard small parts of the graph. For vertex subsets $A, B$ of a graph $G = (V, E)$, we write $w_v(A) = |A|/|V|$, $w_e(A) = |E(A)|/|E|$, and $w_e(A, B) = |E(A, B)|/|E|$. 

Observe that \Cref{thm_strong_const} is about graphs with very small $\lambda_2$. A critical property of such graphs, as shown in \Cref{lem_indep_set} below, is that the union of any two sufficiently large vertex sets with few edges between them is a nearly independent set. Crucially, we do not even any assumptions on the number of edges within each vertex set to reach this conclusion.

\begin{lemma}\label{lem_indep_set}
	Let $G = (V, E)$ be a graph with second eigenvalue $\lambda_2$. For disjoint vertex subsets $A, B$ with $w_v(A), w_v(B) \geq \gamma$ and $w_e(A, B) \leq \xi$, we have
	\[
		w_e(A \cup B) \leq \frac{3 \max\{\lambda_2, \xi\}}{\gamma}.
	\]
\end{lemma}
\begin{proof}
	Let $x \in \R^{V}$ be given by 
	\[
		x_v = \begin{cases}
			w_v(B) &\quad v \in A \\
			-w_v(A) &\quad v \in B \\
			0 &\quad \text{otherwise}.
		\end{cases}
	\]
	Then $x_v \perp \mathbf{1}$. By \Cref{lem_courant}, 
	\begin{align*}
		\lambda_2
        &\geq \frac{\frac{1}{d} \left( 2|E(A)| w_v(B)^2 + 2|E(B)| w_v(A)^2 - 2|E(A, B)| w_v(A)w_v(B)  \right)}{w_v(B)^2 |A| + w_v(A)^2 |B|}\\
		&= \frac{w_e(A) w_v(B)^2 + w_e(B) w_v(A)^2 - w_e(A, B) w_v(A)w_v(B)}{(w_v(A) + w_v(B)) w_v(A) w_v(B)} \\
        &\geq \frac{(w_e(A) + w_e(B)) \gamma - \xi}{(w_v(A) + w_v(B))} \\
        &\geq (w_e(A) + w_e(B)) \gamma - \xi.
	\end{align*}
    As $w_e(A \cup B) = w_e(A) + w_e(B) + w_e(A, B)$, we have
    \[
        w_e(A \cup B) \leq \frac{\lambda_2 + \xi}{\gamma} + \xi \leq \frac{3 \max\{\lambda_2, \xi\}}{\gamma}.
    \]
\end{proof}

Given a set of $\delta$-distinct colorings of $G$, we use \Cref{lem_indep_set} to partition most vertices of $V$ into ``nearly independent sets'' such that no coloring gives the same color to vertices of different nearly independent sets. This allows us to upper bound the distance between colorings on each of these nearly independent sets.

To construct these nearly independent sets, we group the vertices of $G$ using the colorings. Fix $\epsilon > 0$ and let $\delta = 1 - \frac{1}{2} \left( \frac{1}{\lfloor \frac{q}{2} \rfloor} + \frac{1}{\lceil \frac{q}{2} \rceil} \right) + \epsilon$. Suppose that $C = \{X_1, \ldots, X_K\}$ is a $\delta$-distinct set of colorings of $G$. For $\alpha = (a_1, \ldots, a_K) \in [q]^K$, let
\[
    S_{\alpha} = \{v \in V | X_i(v) = a_i \text{ for }i \in [K]\}
\]
be the set of vertices that are colored according to $\alpha$ by the colorings in $C$. Observe that each $S_{\alpha}$ is an independent set, and for any $\alpha, \beta \in [q]^K$ that agree on at least one coordinate, there are no edges between $S_{\alpha}$ and $S_{\beta}$. 

An equivalent way to view the above setup is a graph homomorphism from $\phi: G \to K_{q}^{\otimes K}$, the graph on $[q]^K$ where $(a_1, \ldots, a_K)$ and $(b_1, \ldots, b_K)$ are adjacent if and only if $a_i \neq b_i$ for each $i$. The homomorphism is given by $\phi(v) = (X_1(v), \ldots, X_K(v))$. Then the preimage of a vertex $\alpha$ of $K_{q}^{\otimes K}$ is precisely $S_{\alpha}$. We write $\alpha \sim \beta$ if they are adjacent in $K_{q}^{\otimes K}$; that is, if they don't agree on any coordinate. Thus, we can abstract $G$ as a weighted version of $K_{q}^{\otimes K}$, where for vertex subsets $A, B \subset [q]^K$, we have $w_v(A) = w_v(\phi^{-1}(A))$ and $w_e(A, B) = w_e(\phi^{-1}(A), \phi^{-1}(B))$. For convenience, when $A = \{\alpha\}, B= \{\beta\}$, we will simply write $w_v(\alpha)$ and $w_e(\alpha, \beta)$. 

We are now ready to construct the nearly independent sets from this abstraction of $G$. 
\begin{lemma}\label{lem_partition}
    Suppose $\gamma > 0$ and let $V_{\gamma} = \{\alpha \in [q]^{K} | w_v(\alpha) \geq \gamma\}$. Then there exists a partition of $V_{\gamma}$ into subsets $S_1, \ldots, S_t$ such that 
    \[
        w_e(S_i) \leq \left(\frac{3}{\gamma}\right)^{q^K} \lambda_2
    \]
    for all $i$. Furthermore, for all $\alpha \in S_i$, $\beta \in S_j$ with $i \neq j$, we have that $\alpha \sim \beta$. 
\end{lemma}
\begin{proof}
    We build the $S_i$ sequentially. When we start $S_{i}$, first add an arbitrary remaining element of $V_{\gamma}$ to it. Next, whenever there is some remaining element $\beta \in V_{\gamma}$ such that $\beta \nsim \alpha'$ for some $\alpha' \in S_i$, we add $\beta$ to $S_i$. This way, when we are done with $S_i$, all remaining elements of $V_{\gamma}$ are adjacent to all elements of $S_i$. Thus, this process produces a partition such that, for all $\alpha \in S_i$, $\beta \in S_j$ with $i \neq j$, we have $\alpha \sim \beta$. 

    We now prove that each $S_i$ is nearly independent. Let $S = S_i$ be the current subset we are building in the partition process. We will prove that at any stage, 
    \[
        w_e(S) \leq \left(\frac{3}{\gamma}\right)^{|S|} \lambda_2,
    \]
    which implies the lemma. Initially, when there is only a single element in $S$, we have $w_e(S) = 0$ which clearly satisfies the condition. Now suppose we make the update $S' = S \cup \{\beta\}$ in the process. Let
    \[
        A = \{\beta\} \cup \{\alpha \in S | \alpha \sim \beta\} 
        \quad \text{and} \quad
        B = \{\alpha \in S | \alpha \nsim \beta \}
    \]
    be a partition of $S'$. By construction, both $A$ and $B$ are nonempty and thus have weight at least $\gamma$. Also, since there are no edges between $\beta$ and $B$, we have that $w_e(A, B) \leq w_e(S)$. Thus, by \Cref{lem_indep_set} and the induction hypothesis,
    \[
        w_e(S') \leq \frac{3 \max \{\lambda_2, w_e(S)\}}{\gamma} \leq \left(\frac{3}{\gamma}\right)^{|S| + 1} \lambda_2,
    \]
    as desired.
\end{proof}

One simple constraint on these nearly independent sets is that they cannot get much larger than half of $|V|$.
\begin{lemma}\label{lem_indep_size}
    Let $G = (V, E)$ be a graph and let $A \subset V$. For any $\epsilon > 0$, if $w_e(A) \leq \epsilon$, then $w_v(A) \leq \frac{1 + \epsilon}{2}$. 
\end{lemma}
\begin{proof}
    Since $G$ is a regular graph,
    \[
        2 w_v(A) = 2 w_e(A) + w_e(A, V \setminus A) \leq 1 + w_e(A) \leq 1 + \epsilon. 
    \]
\end{proof}

We are now in a position to prove \Cref{thm_strong_const}. We first take the partition of $V_{\gamma}$ into $S_1, \ldots, S_t$ provided by \Cref{lem_partition}. Using a pigeonhole argument, we show that if $G$ has too many $\delta$-distinct colorings, then we can find many colorings such that for each $S_i$, these colorings use the same set of colors. Then we deduce that some pair of these colorings cannot be $\delta$-distinct to obtain a contradiction. For the last part of the proof, we use the following simple technical lemma. 
\begin{lemma}\label{lem_diagonalize}
    For any $\epsilon > 0$ and any integer $L \geq 1 + 1/\epsilon$, the following holds. Let $w: [q]^L \to \R_{\geq 0}$ be any weight function with $w([q]^L) = 1$ and let $T_1, \ldots, T_t$ be a partition of $[q]$. Then, there exist distinct $a, b \in [L]$ such that 
    \[
        \sum_{\alpha: \alpha_a = \alpha_b} w(\alpha) \geq \sum_{i = 1}^{t} \frac{w(T_i^{L})}{|T_i|} - \epsilon.
    \]
\end{lemma}
Essentially, this lemma says that for a sufficiently large $L$, there is a way to project $[q]^L$ onto $[q]^2$ such that the preimage of the diagonal almost has the expected amount of weight if $w$ was uniform on each of the hypercubes $T_i^{L}$. 
\begin{proof}
    Let $a, b \in [L]$ to be uniformly random pair of distinct elements. Without loss of generality, suppose $T_1 = [T_1] \subset [q]$. Let $\alpha \in T_1^{L}$ and suppose $\alpha$ has $L_1, \ldots, L_{|T_1|}$ coordinates equal to $1, \ldots, T_1$, respectively. Then
    \begin{align*}
        \Pr[\alpha_a = \alpha_b] 
        &= \frac{\sum_{i = 1}^{|T_1|} \binom{L_i}{2}}{\binom{L}{2}}\\
        &\geq \frac{\frac{L^2}{|T_1|} - L}{L^2 - L}\\
        &\geq \frac{1}{|T_1|} - \frac{1}{L - 1}\\
        & \geq \frac{1}{|T_1|} - \epsilon. 
    \end{align*}
    Thus, summing over all $\alpha \in T_i^{L}$ for some $i$, we have
    \begin{align*}
        \E \left[\sum_{\alpha: \alpha_a = \alpha_b} w(\alpha) \right]
        &\geq \sum_{i = 1}^{t} \sum_{\alpha \in T_i^{L}} w(\alpha) \Pr[\alpha_a = \alpha_b]\\
        &\geq \sum_{i = 1}^{t} \frac{w(T_i^{L})}{|T_i|} - \epsilon \sum_{i = 1}^{t} w(T_i^{L})\\
        &\geq \sum_{i = 1}^{t} \frac{w(T_i^{L})}{|T_i|} - \epsilon.
    \end{align*}
    Thus, some choice of $a, b$ produces the desired result. 
\end{proof}

We now prove \Cref{thm_strong_const}.
\begin{proof}[Proof of \Cref{thm_strong_const}]
    We claim that 
    \[
    K(\epsilon) \leq \left\lceil\frac{4q}{\epsilon} \right\rceil \cdot \max_{t} t! S(q, t),
    \]
    where $S(q, t)$ is the Stirling number of the second kind, which counts the number of ways to partition $[q]$ into $t$ nonempty subsets. Let $K = 1 + \lceil\frac{4q}{\epsilon}\rceil \cdot \max_{t} t! S(q, t)$. Let $\gamma = \epsilon/(2q^{K})$. Let $V_{\gamma} = \bigcup_{i = 1}^{t} S_i$ be the partition given by \Cref{lem_partition}. Then $w_v(V_{\gamma}) \geq 1 - q^K \gamma = 1 - \epsilon/2$. By \Cref{lem_indep_size}, we have that $t \geq 2$. Let $\lambda = \lambda(\epsilon)$ be such that 
    \[
        w_e(S_i) \leq \left( \frac{3}{\gamma} \right)^{q^K} \lambda_2 \leq \epsilon
    \]
    for any $\lambda_2 \leq \lambda$. 
    
    Let $S_{ij} \subset [q]$ be the projection of $S_i$ onto the $j$-th coordinate for $i \in [t]$, $j \in [K]$. Since no elements of different parts share a coordinate, for each $j$, the sets $S_{1j}, \ldots, S_{tj}$ are disjoint. We add the remaining elements of $[q]$ to $S_{1j}$ such that $S_{1j}, \ldots, S_{tj}$ partition $[q]$ for each $j$. Since there are at most $t! S(q, t)$ distinct such partitions, the pigeonhole principle implies that there exist a set of at least $L \geq 1 + \frac{4q}{\epsilon}$ colorings which have the exact same corresponding partition. Without loss of generality, suppose that $[L] \subset [K]$ are $L$ such colorings. 

    We take the corresponding projection $\Pi: [q]^K \to [q]^L$ and write $w_v(S) = w_v(\Pi^{-1}(S))$ for any $S \subset [q]^L$. Let $T_1, \ldots, T_t$ be the common partition of $[q]$ by each of the $L$ colorings. Then, by \Cref{lem_diagonalize}, there exist $a, b \in [L]$ such that 
    \[
        \sum_{\alpha: \alpha_a = \alpha_b} w_v(\alpha) \geq \sum_{i = 1}^{t} \frac{w_v(T_i^{L})}{|T_i|} - \frac{\epsilon}{4q}.
    \]
    We now claim that $d_H(X_a, X_b)/n < \delta$. Indeed, we have that 
    \[
        \frac{d_H(X_a, X_b)}{n} = 1 - \sum_{\alpha: \alpha_a = \alpha_b} w_v(\alpha) \leq 1 - \sum_{i = 1}^{t} \frac{w_v(T_i^{L})}{|T_i|} + \frac{\epsilon}{4q}.
    \]
    Recall that $2 \leq t \leq q$, $\sum |T_i| = q$, $\sum w_v(T_i^L) = w_v(V_{\gamma}) \geq 1 - \epsilon/2$, and $w_v(T_i^L) \leq \frac{1 + \epsilon}{2}$ for each $i$ by \Cref{lem_indep_size}. 

    We wish to lower bound the quantity $\sum \frac{w_v(T_i^{L})}{|T_i|}$. Thus, we may assume that $\sum w_v(T_i^L) = 1 - \epsilon/2$. Suppose that $|T_1| \geq \cdots \geq |T_t|$. Then the quantity is minimized by maximizing $w_v(T_1^{L})$ and then $w_v(T_2^{L})$. Therefore, we have
    \[
        \sum_{i = 1}^{t} \frac{w_v(T_i^{L})}{|T_i|} 
        \geq \frac{\frac{1 + \epsilon}{2}}{|T_1|} + \frac{\frac{1 - 2\epsilon}{2}}{|T_2|} 
        \geq \frac{\frac{1 + \epsilon}{2}}{\lceil \frac{q}{2} \rceil} + \frac{\frac{1 - 2\epsilon}{2}}{\lfloor \frac{q}{2} \rfloor}
    \]
    for sufficiently small $\epsilon$. Finally, we have 
    \begin{align*}
        \frac{d_H(X_1, X_2)}{n} 
        &\leq 1 - \left( \frac{\frac{1 + \epsilon}{2}}{\lceil \frac{q}{2} \rceil} + \frac{\frac{1 - 2 \epsilon}{2}}{\lfloor \frac{q}{2} \rfloor} \right) + \frac{\epsilon}{4q}\\
        &\leq 1 - \frac{1}{2} \left(\frac{1}{\lfloor \frac{q}{2} \rfloor} + \frac{1}{\lceil \frac{q}{2} \rceil}\right)  + \epsilon \left(\frac{1}{\lfloor \frac{q}{2} \rfloor} - \frac{1}{2} \frac{1}{\lceil \frac{q}{2} \rceil} + \frac{1}{4q} \right)\\
        &< 1 - \frac{1}{2} \left(\frac{1}{\lfloor \frac{q}{2} \rfloor} + \frac{1}{\lceil \frac{q}{2} \rceil}\right)  + \epsilon \\
        &= \delta.
    \end{align*}
\end{proof}
An interesting phenomenon is that \Cref{thm_strong_const} holds for $g_{\delta, \lambda, d}(n)$ with a constant $K = K(q)$ that is independent of $\epsilon$. We leave as an open problem to prove \Cref{thm_strong_const} with a constant $K$ independent of $\epsilon$ or to demonstrate a sequence of points in the constant regime for $f$ where $f$ takes arbitrarily large constant value. 
\begin{corollary}\label{cor_const}
    There exists an integer $K = K(q)$ such that the following holds. For any $\epsilon > 0$, there exists some $\lambda = \lambda(\epsilon) > 0$ such that for $\delta = 1 - \frac{1}{2} \left( \frac{1}{\lfloor \frac{q}{2} \rfloor} + \frac{1}{\lceil \frac{q}{2} \rceil} \right) + \epsilon$, we have
	\[
		\sup_{d \geq 0} \limsup_{n \to \infty} g_{\delta, \lambda, d}(n) \leq K.
	\]
    In particular, $K(3) = 2$. 
\end{corollary}
\begin{proof}
    We claim that $K(q) \leq \binom{q - 1}{\lfloor \frac{q - 1}{2} \rfloor}$, so that in particular $K(3) \leq 2$. The gain here comes from the fact that we get to arbitrarily permute the colors when trying to upper bound $d_c$. Thus, it suffices to find a pair of colorings $a, b \in [K]$ such that $|S_{1a}| = |S_{1b}|, \ldots, |S_{ta}| = |S_{tb}|$. This only requires $K \geq \binom{q - 1}{t - 1}$, which produces our bound. 

    After we find this pair, we also bypass \Cref{lem_diagonalize} altogether, since we can permute the colors so that at least the expected amount of weight lie on the diagonal, getting rid of the $\frac{\epsilon}{4q}$ error term. The rest of the computation follows exactly. 
\end{proof}

\subsection{Proof of \Cref{thm_const}~(2)}\label{subsec_4.2}
We now turn to the case when $\delta = 1 - \frac{1}{q}$. We first give the construction for the negative result in \Cref{thm_const}~(2). 
\begin{lemma}\label{lem_non_const}
    The point $\left(1- \frac{1}{q}, \frac{1}{(q - 1)^2} \right)$ is not in the constant regime for $g$.  
\end{lemma}
To prove \Cref{lem_non_const}, we construct a family of graphs $G = G_n$ with $\lambda_2 = \frac{1}{(q - 1)^2}$ such that $G$ has a strongly $\left( 1- \frac{1}{q} \right)$-distinct set of colorings of arbitrary size. We obtain the sequence $G_n$ by constructing repeated \emph{$2$-lifts} of $G_1$. Recall that a $2$-lift of a graph is defined in the following way.

Let $G = (V, E)$ be an $n$-vertex graph. A \emph{signing} of the edges of $G$ is a function $s: E \to \{-1, 1\}$. The \emph{$2$-lift} of $G$ associated with signing $s$ is a graph $\widehat{G} = (\widehat{V}, \widehat{E})$ on $2n$ vertices given as follows. For each $v \in V$, we have two vertices $v_1, v_2 \in \widehat{V}$ associated to $v$. For each $uv \in E$, we have two edges in $\widehat{E}$ associated to it. If $s(uv) = 1$, then the two edges are $u_1 v_1$ and $u_2 v_2$. If $s(uv) = -1$, then the two edges are $u_1 v_2$ and $u_2 v_1$.

Observe that if $G$ is $d$-regular, then any $2$-lift of $G$ is also $d$-regular. The following result of Bilu and Linial \cite{bilu2006} shows that we can always find a $2$-lift that preserves the second eigenvalue of $G$. 
\begin{lemma}\cite{bilu2006}\label{lem_2_lift}
    There exists a constant $C > 0$ such that the following holds. Let $G$ be a $d$-regular graph with normalized second eigenvalue $\lambda_2$. There exists a $2$-lift of $G$ whose normalized second eigenvalue is
    \[  
        \max \left\{\lambda_2, C\left(\sqrt{\frac{\log^3 d}{d}} \right)\right\}.
    \]
\end{lemma}
Now we prove \Cref{lem_non_const}.
\begin{proof}[Proof of \Cref{lem_non_const}]
    Let $N \geq 2$ be any sufficiently large positive integer. Let $G_1 = K_q^{\otimes N}$ be the graph on $[q]^N$ where $(a_1, \ldots, a_N)$ and $(b_1, \ldots, b_N)$ are adjacent if and only if $a_i \neq b_i$ for all $i \in [N]$. Then $G_1$ is $(q - 1)^N$-regular. Furthermore, the normalized adjacency matrix of $G_1$ is
    \[
        A = A_{K_q}^{\otimes N},
    \]
    where $A_{K_q}$ is the normalized adjacency matrix of $K_q$ and $\otimes$ is the Kronecker product. Since the spectrum of $K_q$ is given by $\lambda_1 = 1$ and $\lambda_2 = \cdots = \lambda_q = -\frac{1}{q - 1}$, we have that $\lambda_2(G_1) = \frac{1}{(q - 1)^2}$. By \Cref{lem_2_lift}, for $N$ sufficiently large, we can find a sequence of repeated $2$-lifts of $G_1$ to form a sequence $G = G_n$ of $(q - 1)^N$-regular graphs with $\lambda_2 = \frac{1}{(q - 1)^2}$.

    We show that each $G_n$ in this sequence has a strongly $\left(1 - \frac{1}{q} \right)$-distinct set of colorings of size $N$. First, consider the colorings $X_1, \ldots, X_N$ of $G_1$ where $X_i(a_1, \ldots a_N) = a_i$. It is easy to verify that $d_c(X_i, X_j) = \left(1 - \frac{1}{q} \right) q^K$. 
    
    Now it suffices to show that a strongly $\delta$-distinct set of colorings of a graph can be lifted to a strongly $\delta$-distinct set of colorings of its $2$-lift. Indeed, if $X_1, \ldots, X_N$ are colorings of $G = (V, E)$, then consider colorings $\widehat{X_1}, \ldots, \widehat{X_N}$ where $\widehat{X_i}(v_1) = \widehat{X_i}(v_2) = X_i(v)$ for all $v \in V$. It is clear that the coloring distance between any pair of colorings is doubled, so the parameter $\delta$ is preserved. 

    Therefore, there exists a sequence $G = G_n$ of $(q - 1)^N$-regular graphs with $\lambda_2 = \frac{1}{(q - 1)^2}$ such that each $G_n$ has a strongly $\left(1 - \frac{1}{q} \right)$-distinct set of colorings of size $N$. Taking $N$ to infinity, we have that $\left(1 - \frac{1}{q}, \frac{1}{(q - 1)^2}\right)$ is not in the constant regime. 
\end{proof}

Now it remains to show that, by letting $\lambda$ be a little smaller than $1/(q - 1)^2$, we return to the constant regime for both $f$ and $g$. We make use of the following general lemma. 
\begin{lemma}\label{lem_test_vec}
    Let $G = (V, E)$ be a regular graph with $\lambda_2 \leq \lambda$ and let $S \subset V$ be a vertex subset. Then 
    \[
        w_e(S, \overline{S}) \geq 2(1 - \lambda)w_v(S)(1 - w_v(S)).
    \]
\end{lemma}
\begin{proof}
    Consider the test vector $x \in \R^{V}$ given by 
    \[
        x_v = \begin{cases}
            w_v(S) - 1 \quad & v \in S\\
            w_v(S) \quad & v \in \overline{S}.
        \end{cases}
    \]
    Then $x \perp \mathbf{1}$. Recall that for any $A \subset V$, $w_e(A) + \frac{1}{2}w_e(A, \overline{A}) = w_v(A)$. By \Cref{lem_courant}, we have 
    \begin{align*}
        \lambda 
        \geq \lambda_2
        &\geq \frac{\frac{1}{d}(2|E(S)|(w_v(S) - 1)^2 + 2|E(\overline{S})|w_v(S)^2 + 2|E(S, \overline{S})|(w_v(S) - 1)w_v(S) )}{(w_v(S) - 1)^2 |S| + w_v(S)^2 |\overline{S}|}\\
        &= \frac{w_e(S)(1 - w_v(S))^2 + w_e(\overline{S})w_v(S)^2 - w_e(S, \overline{S}) (1 - w_v(S))w_v(S)}{w_v(S)(1 - w_v(S))}\\
        &= \frac{w_v(S) (1 - w_v(S))^2 + (1 - w_v(S))w_v(S)^2 - \frac{1}{2}w_e(S, \overline{S})}{w_v(S)(1 - w_v(S))}\\
        &= 1 - \frac{w_e(S, \overline{S})}{2w_v(S)(1 - w_v(S))}.
    \end{align*}
    Rearranging gives the desired inequality. 
\end{proof}
We now prove the remaining part of \Cref{thm_const}. 
\begin{theorem}\label{thm_const_part_2}
    For any $\epsilon > 0$, $\left(1 - \frac{1}{q}, \frac{1}{(q - 1)^2} - \epsilon \right)$ is in the constant regime for $f$.
\end{theorem}
\begin{proof}
    Let $K = \lceil 1 + 2q^2/\epsilon \rceil$. Let $\delta = 1 - \frac{1}{q}$ and $\lambda = \frac{1}{(q - 1)^2} - \epsilon$. We will prove that $f_{\delta, \lambda, d}(n) < K$. Suppose for the sake of contradiction that $G$ is a $d$-regular $n$ vertex graph with $\lambda_2 \leq \lambda$ with a set $\{X_1, \ldots, X_K\}$ of $K$ $\delta$-distinct colorings. Recall the identification of $G$ with a weighted version of $K_q^{\otimes K}$ given in \Cref{subsec_4.1}. 

    For any two colorings $X_a, X_b$, observe that 
    \[
        1 - \frac{d_H(X_a, X_b)}{n} = w_v(S_{ab}) \leq \frac{1}{q},
    \]
    where $S_{ab} = \{\alpha \in [q]^K | \alpha_a = \alpha_b\}$. We claim that there exist $a, b \in [K]$ such that 
    \[
        w_e(S_{ab}) \geq \frac{1}{q(q-1)} - \frac{1}{K - 1}.
    \]
    Similar to the proof of \Cref{lem_diagonalize}, we choose $a, b$ to be a uniformly random pair of distinct numbers. Let $\alpha, \beta \in [q]^K$ and $\alpha \sim \beta$ (each coordinate of $\alpha$ differs from $\beta$). Let $C_{ij}$ be the number of colorings that color $\alpha$ by $i$ and $\beta$ by $j$. Note that $C_{ii} = 0$ for all $i$. Then
    \begin{align*}
        \Pr[\alpha, \beta \in S_{ab}] 
        &= \frac{\sum_{i \neq j} \binom{C_{ij}}{2}}{\binom{K}{2}} \\
        &\geq \frac{\frac{K^2}{q(q-1)} - K}{K^2 - K}\\ 
        &\geq \frac{1}{q(q-1)} - \frac{1}{K - 1}.
    \end{align*}
    Thus, for any $\alpha, \beta$ such that there could be edges between them, the probability that they are included in $w_e(S_{ab})$ is at least $\frac{1}{q(q-1)} - \frac{1}{K - 1}$, which implies that the expectation of $w_e(S_{ab})$ is at least this. This proves our claim. In particular, this means that $w_e(S_{ab}) \geq \frac{1}{q - 1} w_v(S_{ab}) - \frac{1}{K - 1}$. 

    Since $w_e(S_{ab}, \overline{S_{ab}}) = 2(w_v(S_{ab}) - w_e(S_{ab}))$, we have that 
    \[
        w_e(S_{ab}, \overline{S_{ab}}) \leq 2 \left(1 - \frac{1}{q - 1}\right) w_v(S_{ab}) + \frac{2}{K - 1}.
    \]
    On the other hand, \Cref{lem_test_vec} gives
    \[
        w_e(S_{ab}, \overline{S_{ab}}) 
        \geq 2(1 - \lambda)w_v(S_{ab})(1 - w_v(S_{ab}))
        \geq 2(1 - \lambda)w_v(S_{ab})\left(1 - \frac{1}{q}\right).
    \]
    Combining the two inequalities, we have
    \begin{align*}
        \frac{\epsilon}{2q^2} 
        &\geq \frac{1}{K - 1}\\
        &\geq \left(1 - \frac{1}{(q-1)^2} + \epsilon \right)
        \left(1 - \frac{1}{q}\right) w_v(S_{ab})
        - \left(1 - \frac{1}{q - 1}\right) w_v(S_{ab})\\
        & = \epsilon  \left(1 - \frac{1}{q}\right) w_v(S_{ab})\\
        &\geq \epsilon  \left(1 - \frac{1}{q}\right) \left( \frac{1}{q(q-1)} - \frac{1}{K - 1} \right)\\
        & \geq \frac{\epsilon}{q^2} - \frac{\epsilon}{K - 1}\\
        & \geq \frac{\epsilon}{q^2} - \frac{2 \epsilon^2}{q^2},
    \end{align*}
    which is false for all $\epsilon < 1/4$. 
\end{proof}

\section{The unique regime}\label{sec_unique_regime}
We prove \Cref{thm_unique} in this section. We first show the inclusion results for the unique regime. This extends Lemma~5.7 of \cite{bafna2024} which focuses on the case $q = 3$. 
\begin{theorem}\label{thm_strong_unique}
    Any $(\delta, \lambda) \in [1 - 1/(q - 1), 1 - 1/q] \times (0, 1]$ satisfying 
    \[
        (q-1)(1-\delta)^2 + (1 - (q-1)(1 - \delta))^2 < 1 - \frac{1 - \frac{1}{q - 1}}{1 - \lambda}
    \]
    is in the unique regime. In particular, for any $\epsilon > 0$, $\left(1 - \frac{1}{q}, \frac{1}{(q-1)^2} - \epsilon \right)$ is in the unique regime, and there exists some $\lambda = \lambda(\epsilon) > 0$ such that $\left( 1 - \frac{1}{q - 1} + \epsilon, \lambda \right)$ is in the unique regime. 
\end{theorem}
\begin{proof}
    Suppose that $G = (V, E)$ is an $n$ vertex $d$-regular graph with $\lambda_2 \leq \lambda$ and $q$-colorings $X$ and $Y$ satisfying $d_c(X, Y) \geq \delta n$. Furthermore, assume that $\delta \geq 1 - 1/(q - 1)$. For each $\sigma \in S_q$, let 
    \[
        V_{\sigma} = \{v \in V | X(v) = \sigma(Y(v))\}.
    \]
    We have $d_c(X, Y) = n - \max_{\sigma} |V_{\sigma}|$ and thus $w_v(V_{\sigma}) \leq 1 - \delta \leq 1/(q - 1)$ for all $\sigma$.
    
    By \Cref{lem_test_vec}, we have 
    \[
        w_e(V_{\sigma}, \overline{V_{\sigma}}) \geq 2 (1 - \lambda) (w_v(V_{\sigma}) - w_v(V_{\sigma})^2). 
    \]
    Observe that each vertex $v \in V$ appears in $V_\sigma$ for exactly $(q - 1)!$ different $\sigma \in S_q$ and each edge $uv \in E$ appears in $E(V_{\sigma}, \overline{V_{\sigma}})$ for exactly $2 (q - 2)(q - 2)!$ different $\sigma \in S_q$. Therefore, summing over all $\sigma \in S_q$, we have 
    \[
        2 (q - 2)(q - 2)! \geq 2(1 - \lambda) \left((q - 1)! - \sum_{\sigma \in S_q} w_v(V_{\sigma})^2 \right).
    \]
    Rearranging, we have 
    \[
        \frac{1}{(q-1)!} \sum_{\sigma \in S_q} w_v(V_{\sigma})^2 \geq 1 - \frac{1 - \frac{1}{q - 1}}{1 - \lambda}. 
    \]
    Let $\pi = (1 \, 2 \, \cdots \, q) \in S_q$ be the cyclic shift. Then for any $\sigma \in S_q$, 
    \[
        V = \bigcup_{i = 0}^{q - 1} V_{\pi^i \sigma}.
    \]
    Since $w_v(V_{\sigma}) \leq 1/(q - 1)$, 
    \[
        \sum_{i = 0}^{q - 1} w_v(V_{\pi^i \sigma})^2 \leq (q - 1)(1 - \delta)^2 + (1 - (q - 1)(1 - \delta))^2.
    \]
    Since $S_q$ can be partitioned into $(q - 1)!$ such orbits, we have that 
    \[
        (q - 1)(1 - \delta)^2 + (1 - (q - 1)(1 - \delta))^2 \geq 1 - \frac{1 - \frac{1}{q - 1}}{1 - \lambda}.
    \]
    Thus, if the inequality does not hold for $(\delta, \lambda)$, it means that $G$ cannot have two $q$-colorings that are $\delta$-distinct. 
\end{proof}
Finally, we give the construction for the negative result in \Cref{thm_unique}~(1). 
\begin{lemma}\label{lem_non_unique}
    For all $\lambda > 0$, $\left( 1- \frac{1}{q - 1}, \lambda \right)$ is not in the unique regime. 
\end{lemma}
\begin{proof}
    Fix $\lambda > 0$ and choose $d$ such that $2 / \sqrt{d} \leq \lambda$. Let $G = G_n$ be a sequence of $d$-regular bipartite graphs on $2(q - 1)n$ vertices. By \Cref{lem_eigenval_of_random_bipartite}, we can choose $G_n$ such that $\lambda_2(G_n) \leq 2 / \sqrt{d} \leq \lambda$ for sufficiently large $n$. 

    We now color $G$ in two ways. Let $X$ be a coloring that colors the first part of $G$ by $1$ and the second part by $2, \ldots, q$ such that each color is used on exactly $n$ vertices. Similarly, let $Y$ be a coloring that colors the second part of $G$ by $q$ and the first part of $G$ by $1, \ldots, q - 1$ such that each color is used on exactly $n$ vertices. Then 
    \[
        d_c(X, Y) = 2(q - 2)n = \left( 1 - \frac{1}{q - 1} \right) |V(G)|.
    \]
    Therefore, $\left( 1- \frac{1}{q - 1}, \lambda \right)$ is not in the unique regime.
\end{proof}

\section{Conclusion and future directions}\label{sec_future}
We introduced a new variant of error-correcting codes, which we call graphical error-correcting codes, that generalizes typical error-correcting codes. This problem can be viewed as a sphere packing problem on the space of proper $q$-colorings of a graph, which reduces to the typical sphere packing problem on $[q]^n$ equipped with the Hamming distance when the graph is empty. 

We believe that graphical error-correcting codes provides new insights for the fields of coding theory as well as extremal graph theory by presenting a novel connection between them. In the context of this paper, we studied the problem in the particular case of expander graphs, which has already led to useful applications. For example, Bafna et al. \cite{bafna2024} proved a result analogous to \Cref{thm_strong_unique} in the case $q = 3$ to construct an efficient algorithm for finding a large independent set in a $3$-colorable expander, while it is known to be NP-hard (assuming the unique-games conjecture) to find a large independent set in a general $3$-colorable graph.

One fundamental open question we want to address is the relationship between the Hamming distance $d_H$ and the coloring distance $d_c$. Inspired by the work of Bafna et al. \cite{bafna2024}, we originally viewed the ``distance'' between two graph colorings as ``color blind,'' leading to our definition of $d_c$. However, in the context of error-correcting codes, the Hamming distance is more natural. It turns out that up to constant factors (possibly depending on $\delta, \lambda, d, q$), all of our theorems about the exponential and constant regimes hold under either notion of distance. While the two distances are clearly not equivalent up to constant factors, there seems to be an intimate relationship between sets of distinct colorings and sets of strongly distinct colorings. An example of a question we can ask is the following.
\begin{question}
    Clearly, $R_{f, exp} \supset R_{g, exp}$ and $R_{f, con} \subset R_{g, con}$. Are either relations an equality?
\end{question}

Another natural future direction is to understand the landscape of the regimes in \Cref{fig_regimes} better and fill in the gaps. In particular, does $f$ and $g$ have an intermediate growth behavior for some $(\delta, \lambda)$? 
\begin{question}
    Is there some $(\delta, \lambda) \in [0, 1 - 1/q] \times (0, 1]$ that is neither in the exponential regime nor the constant regime for $f$ or $g$? If so, what growth behavior do they exhibit as $n \to \infty$?
\end{question}

It appears that when $q$ is large, the exponential regime should be much larger than what \Cref{thm_bipartite_exp,thm_exp} guarantee. Indeed, for $d < q - 1$, any $d$-regular graph on $n$ vertices has at least $(q - 1 - d)^n$ proper $q$-colorings. We expect that there are good expanders for which these colorings are distributed sufficiently ``uniformly'' over $[q]^n$. 

\begin{conjecture}\label{conj_exp}
    For all $q$ sufficiently large, there exists an infinite sequence of $d$-regular graphs with $\lambda_2 = \Theta(1/\sqrt{d})$ for some $d = \Theta(q)$ such that each graph has a $\delta$-distinct set of proper $q$-colorings of exponential size for any $\delta < 1 - 1/q$. In particular, this implies that $(\delta, \lambda)$ is in the exponential regime for all $\delta < 1 - 1/q$ and some $\lambda = O(1/\sqrt{q})$.
\end{conjecture}

This intuition suggests that only expansion on the order of $o_q(1)$ forces the sizes of $\delta$-distinct sets of colorings to be sub-exponential, and the predominant behavior of the count function $f_{\delta, \lambda, d}(n)$ is exponential, as is the case for typical error-correcting codes. To gain a more refined understanding of the exponential regime, it is thus natural to study the rate question as in the theory of error-correcting codes. 
\begin{question}
For fixed $q \geq 3$ and $(\delta, \lambda, d)$, give bounds for
\[
    R(\delta, \lambda, d) := \limsup_{n \to \infty} \frac{\log_{q} f_{\delta, \lambda, d}(n)}{n}.
\]
In particular, we always have that $0 \leq R(\delta, \lambda, d) \leq 1$.
\end{question}
When $q - 1 > d$, $q$-colorability is guaranteed. Thus, we may also consider the rate question for two-sided expanders as suggested in \Cref{rmk:one-sided}. In particular, for two-sided expanders, we have access to more powerful tools such as the expander mixing lemma. We could also consider the rate question for families of graphs that do not expand, such as graphs with bounded doubling. In this case, the restrictions imposed by the graph on the space of proper colorings are ``local,'' which makes it potentially possible to obtain rates close to those of typical error-correcting codes.

\bibliographystyle{amsplain0}
\bibliography{bib}

\end{document}